\documentclass{amsart} 
\usepackage[top=3cm,left=3.1cm,right=3.1cm,bottom=2.1cm]{geometry}
\usepackage{amsfonts}
\usepackage{amsmath,amsthm,amscd,mathrsfs,amssymb,graphicx,enumerate,
  dsfont}
\usepackage{xypic}

\newtheorem*{thm3}{Theorem}
\newtheorem{thm2}{Theorem}
\newtheorem{cor2}{Corollary}
\newtheorem{conj2}{Conjecture}
\newtheorem{thm}{Theorem}[section]

\newtheorem{prop}[thm]{Proposition}

\theoremstyle{definition}

\newcommand{\C}{\ensuremath\mathds{C}}

\newcommand{\Z}{\ensuremath\mathds{Z}}
\newcommand{\Q}{\ensuremath\mathds{Q}}

\newcommand{\h}{\ensuremath\mathfrak{h}}
\renewcommand{\o}{\ensuremath\mathfrak{0}}

\newcommand{\HH}{\ensuremath\mathrm{H}}
\newcommand{\CH}{\ensuremath\mathrm{CH}}

\begin{document}

\thispagestyle{empty} %% Remove header and footer.
\title{On the motive of some hyperK\"ahler varieties} \author{Charles Vial}



\thanks{}

\thanks{The author is supported by EPSRC Early Career Fellowship
  number EP/K005545/1.}

\address{DPMMS, University of Cambridge, Wilberforce Road, Cambridge
  CB3 0WB, UK} \email{c.vial@dpmms.cam.ac.uk}

\date{\today}

\begin{abstract} We show that the motive of the Hilbert scheme of
  length-$n$ subschemes on a K3 surface or on an abelian surface
  admits a decomposition similar to the decomposition of the motive of
  an abelian variety obtained by Shermenev, Beauville, and Deninger and Murre.
\end{abstract}

\maketitle

\section*{Introduction}

In this work, we fix a field $k$ and all varieties are defined over
this field $k$. Chow groups are always meant
with rational coefficients and $\HH^*(-,\Q)$ is Betti cohomology with
rational coefficients.  Up to replacing Betti cohomology with a suitable Weil
cohomology theory (for example $\ell$-adic cohomology), we may and we will
assume that $k$ is a subfield of the complex numbers $\C$. We use freely the
language of (Chow) motives as
is described in \cite{manin}. \medskip

Work of Shermenev \cite{shermenev}, Beauville \cite{beauville1}, and
Deninger and Murre \cite{dm} unravelled the structure of the motives
of abelian varieties~:

\begin{thm3} [Beauville, Deninger--Murre, Shermenev] Let $A$ be an
  abelian variety of dimension $g$. Then the Chow motive $\h(A)$ of
  $A$ splits as
\begin{equation} \label{eq CK} \h(A) = \bigoplus_{i=0}^{2g} \h^i(A)
\end{equation}
 with the following
  properties~:
\begin{enumerate}[(i)]
\item $\HH^*(\h^i(A),\Q) = \HH^i(A,\Q)$ ;
\item the multiplication $\h(A)\otimes \h(A) \rightarrow \h(A)$
  (\emph{cf.} \eqref{eq mult}) factors through $\h^{i+j}(A)$ when
  restricted to $\h^i(A)\otimes \h^j(A)$~;
\item the morphism $[n]^* : \h^i(A) \rightarrow \h^i(A)$ induced by
  the multiplication by $n$ morphism $[n] : A \rightarrow A$ is
  multiplication by $n^i$. Here, $n$ is an integer. In particular, $\h^i(A)$ is
an eigen-submotive for the action of $[n]$.
\end{enumerate}
\end{thm3}

For an arbitrary smooth projective variety $X$, it is expected that a
decomposition of the motive $\h(X)$ as in \eqref{eq CK} satisfying
\emph{(i)} should exist ; see \cite{murre}. Such a decomposition is
called a \emph{Chow--K\"unneth decomposition}. However, in general,
there is no analogue of the multiplication by $n$ morphisms, and the
existence of a Chow--K\"unneth decomposition of the motive of $X$
satisfying \emph{(ii)} (in that case, the Chow--K\"unneth
decomposition is said to be \emph{multiplicative}) is very
restrictive. We refer to \cite[Section 8]{sv} for some discussion on
the existence of such a multiplicative decomposition.

Nonetheless, inspired by the seminal work of Beauville and Voisin
\cite{bv}, \cite{beauville2} and \cite{voisin2}, we were led
to ask in \cite{sv} whether the motives of hyperK\"ahler varieties
admit a multiplicative decomposition similar to that of the motive of
abelian varieties as in the theorem of Beauville, Deninger and Murre,
and Shermenev. Here, by hyperK\"ahler variety we mean a simply
connected smooth projective variety $X$ whose space of global
$2$-forms $\HH^0(X,\Omega_X^2)$ is spanned by a nowhere degenerate
$2$-form. When $k=\C$, a hyperK\"ahler variety is nothing but a
projective irreducible holomorphic symplectic manifold \cite{beauvillec1}.

\begin{conj2} \label{conj} Let $X$ be a hyperK\"ahler variety of
  dimension $2n$. Then the Chow motive $\h(X)$ of $X$ splits
  as $$\h(X) = \bigoplus_{i=0}^{4n} \h^i(X)$$ with the property that
\begin{enumerate}[(i)]
\item $\HH^*(\h^i(X),\Q) = \HH^i(X,\Q)$ ;
\item the multiplication $\h(X)\otimes \h(X) \rightarrow \h(X)$
  (\emph{cf.} \eqref{eq mult}) factors through $\h^{i+j}(X)$ when
  restricted to $\h^i(X)\otimes \h^j(X)$.
\end{enumerate}
\end{conj2}

An important class of hyperK\"ahler varieties is given by the Hilbert
schemes $S^{[n]}$ of length-$n$ subschemes on a K3 surface $S$ ; see
\cite{beauvillec1}. The following theorem shows in particular that the
motive of $S^{[n]}$ for $S$ a K3 surface admits a decomposition with
properties \emph{(i)} and \emph{(ii)} and thus answers affirmatively
the question raised in Conjecture \ref{conj} in that case.

\begin{thm2} \label{thm main} Let $S$ be either a K3 surface or an
  abelian surface, and let $n$ be a positive integer. Then the Chow
  motive $\h(S^{[n]})$ of $S^{[n]}$ splits as $$\h(S^{[n]}) =
  \bigoplus_{i=0}^{4n} \h^i(S^{[n]})$$ with the property that
\begin{enumerate}[(i)]
\item $\HH^*(\h^i(S^{[n]}),\Q) = \HH^i(S^{[n]},\Q)$ ;
\item the multiplication $\h(S^{[n]})\otimes \h(S^{[n]}) \rightarrow
  \h(S^{[n]})$ factors through $\h^{i+j}(S^{[n]})$ when restricted to
  $\h^i(S^{[n]})\otimes \h^j(S^{[n]})$.
\end{enumerate}
\end{thm2}

Theorem 6 of \cite{sv} can then be improved by including the Hilbert
schemes of length-$n$ subschemes on K3 surfaces. 
Theorem
\ref{thm main} is due for $S$ a K3 surface and $n=1$ to Beauville and
Voisin \cite{bv} (see \cite[Proposition 8.14]{sv} for the link between
the original statement of \cite{bv} (recalled in Theorem \ref{thm bv}) and the
statement given here), and
was established in \cite{sv} for $n=2$. Its proof in full generality
is given in Section 3. Note that, as explained in Section 1, the
existence of a Chow--K\"unneth decomposition for the Hilbert scheme
$S^{[n]}$ of any smooth projective surface $S$ goes back to de Cataldo
and Migliorini \cite{dCM} (the existence of such a decomposition for
$S$ is due to Murre \cite{murre2}).
Our main contribution is the claim
that by choosing the Beauville--Voisin decomposition of K3 surfaces \cite{bv},
the induced Chow--K\"unneth decomposition of Hilbert schemes of K3 surfaces
established by de Cataldo and Migliorini \cite{dCM} is
multiplicative, i.e. it satisfies \emph{(ii)}.\medskip

Let us then define for all $i\geq 0$ and all $s \in \Z$ 
$$\CH^i(S^{[n]})_s := \CH^i(\h^{2i-s}(S^{[n]})).$$
We have the following corollary to Theorem \ref{thm main}~:

\begin{thm2} \label{thm bigrading} The Chow ring $\CH^*(S^{[n]})$
  admits a multiplicative bigrading
 $$\CH^*(S^{[n]}) = \bigoplus_{i,s} \CH^i(S^{[n]})_s$$ 
 that is induced by a Chow--K\"unneth decomposition of the diagonal
 (as defined in \S 1).  Moreover, 
 the Chern classes $c_i(S^{[n]})$ belong to the graded-zero part
 $\CH^i(S^{[n]})_0$ of $\CH^i(S^{[n]})$.
\end{thm2}

Theorem \ref{thm bigrading} answers partially a question raised by
Beauville in \cite{beauville2}~: the  filtration $F^\bullet$ defined by
$\mathrm{F}^l\CH^i(X) :=
\bigoplus_{s\geq l}\CH^i(X)_s$ is a filtration on the Chow ring $\CH^*(S^{[n]})$
that is split. Moreover, this filtration is expected to be the one predicted by 
Bloch and Beilinson (because it is induced
by a Chow--K\"unneth decomposition -- conjecturally all such
filtrations coincide). For this filtration to be of Bloch--Beilinson type, one
would need to establish Murre's conjectures, namely that $\CH^i(S^{[n]})_s = 0$
for $s<0$ and that $\bigoplus_{s>0} \CH^i(S^{[n]})_s$ is exactly the kernel of
the cycle class map $\CH^i(S^{[n]}) \rightarrow \HH^{2i}(S^{[n]},\Q)$. 
Note that for $i=0,1,2n-1$ or $2n$, it is indeed the case that $\CH^i(S^{[n]})_s = 0$
for $s<0$ and that $\bigoplus_{s>0} \CH^i(S^{[n]})_s = \mathrm{Ker} \, \{\CH^i(S^{[n]}) \rightarrow \HH^{2i}(S^{[n]},\Q)\}$. Therefore, we have 
\begin{cor2}\label{cor main}
Let $i_1,\ldots, i_m$ be positive integers such that $i_1 + \cdots + i_m = 2n-1$ or $2n$, and let $\gamma_l$ be cycles in $\CH^{i_l}(S^{[n]})$ for $l=1,\ldots, m$ that sit in $\CH^{i_l}(S^{[n]})_0$ for the grading induced by the decomposition of Theorem \ref{thm main}. Then, $[\gamma_1]\cdot [\gamma_2] \cdots  [\gamma_m] = 0$ in $\HH^*(S^{[n]},\Q)$ if and only if  $\gamma_1\cdot \gamma_2 \cdots \gamma_m = 0$ in $\CH^*(S^{[n]})$.
\end{cor2}
 \medskip

Let us mention that Theorem \ref{thm main} and Theorem \ref{thm
  bigrading} (and a fortiori Corollary \ref{cor main}) are also valid for hyperK\"ahler varieties that are
birational to $S^{[n]}$, for some K3 surface $S$. Indeed, Riess \cite{greiner}
showed that birational hyperK\"ahler varieties have isomorphic Chow rings and
isomorphic Chow motives (as algebras in the category of Chow motives) ; see also
\cite[Section 6]{sv}. As for more
evidence as why Conjecture \ref{conj} should be true, Mingmin Shen and
I showed \cite{sv} that the variety of lines on a very general cubic
fourfold satisfies the conclusions of Theorem \ref{thm
  bigrading}. \medskip

Finally, we use the notion of multiplicative Chow--K\"unneth
decomposition to obtain new decomposition results in the spirit of
\cite{voisin k3} ; see Theorem \ref{thm dec}.

\subsection*{Notations} A morphism denoted $pr_r$ will always denote
the projection on the $r^\mathrm{th}$ factor and a morphism denoted
$pr_{s,t}$ will always denote the projection on the product of the
$s^\mathrm{th}$ and $t^\mathrm{th}$ factors. The context will usually
make it clear which varieties are involved.  Chow groups $\CH^i$ are
with rational coefficients. If $X$ is a variety, the cycle class map
sends a cycle $\sigma \in \CH^i(X)$ to its cohomology class $[\sigma]
\in \HH^{2i}(X,\Q)$.  If $Y$ is another variety and if $\gamma$ is a
correspondence in $\CH^i(X \times Y)$, its transpose ${}^t\gamma \in
\CH^i(Y \times X)$ is the image of $\gamma$ under the action of the
permutation map $X \times Y \rightarrow Y \times X$. If
$\gamma_1,\cdots, \gamma_n$ are correspondences in $\CH^*(X \times
Y)$, then the correspondence $\gamma_1 \otimes \cdots \otimes \gamma_n
\in \CH^*(X^n\times Y^n)$ is defined as $\gamma_1 \otimes \cdots
\otimes \gamma_n := \prod_{i=1}^n (pr_{i,n+i})^*\gamma_i$.

\vspace{10pt}
\section{Chow--K\"unneth decompositions}

A Chow motive $M$ is said to have a \emph{Chow--K\"unneth
  decomposition} if it splits as $M=\bigoplus_{i \in \Z} M^i$ with
$\HH^*(M^i,\Q) = \HH^i(M,\Q)$. In other words, $M$ admits a K\"unneth
decomposition that lifts to rational equivalence. Concretely, if
$M=(X,p,n)$ with $X$ a smooth projective variety of pure dimension $d$
and $p \in \CH^d(X\times X)$ an idempotent and $n$ an integer, then
$M$ has a Chow--K\"unneth decomposition if there exist finitely many
correspondences $p^i \in \CH^d(X \times X)$, $i \in \Z$, such that
$p=\sum_i p^i$, $p^i \circ p^i = p^i$, $p\circ p^i = p^i \circ p =
p^i$, $p^i \circ p^j = 0$ for all $i \in \Z$ and all $j \neq i$ and
such that $p^i_*\HH^*(X,\Q) = p_*\HH^{i+2n}(X,\Q)$.

A smooth projective variety $X$ of dimension $d$ has a Chow--K\"unneth
decomposition if its Chow motive $\h(X)$ has a Chow--K\"unneth
decomposition, that is, there exist correspondences $\pi^i \in \CH^d(X\times X)$
such that $\Delta_X = \sum_{i=0}^{2d} \pi^i$, with
$\pi^i \circ \pi^i = \pi^i$, $\pi^i \circ \pi^j = 0$ for $i \neq j$
and $\pi^i_*\HH^*(X,\Q) = \HH^i(X,\Q)$. A Chow--K\"unneth decomposition
$\{\pi^i, 0 \leq i \leq 2d\}$ of $X$ is said to be \emph{self-dual} if
$\pi^{2d-i} = {}^t\pi^i$ for all $i$.

If $\o$ is the class of a rational point on $X$ (or more generally a
zero-cycle of degree $1$ on $X$), then $\pi^0 := pr_1^*\o = \o \times
X$ and $\pi^{2d} = pr_2^*\o = X \times \o$ define mutually orthogonal
idempotents such that $\pi^0_*\HH^*(X,\Q) = \HH^0(X,\Q)$ and
$\pi^{2d}_*\HH^*(X,\Q) = \HH^{2d}(X,\Q)$.  Note that pairs of
idempotents with the property above are certainly not unique~: a
different choice (modulo rational equivalence) of zero-cycle of degree
$1$ gives different idempotents in the ring of correspondences
$\CH^d(X \times X)$. From the above, one sees that every curve $C$
admits a Chow--K\"unneth decomposition~: one defines $\pi^0$ and
$\pi^2$ as above and then $\pi^1$ is simply given by $\Delta_C - \pi^0
- \pi^2$. It is a theorem of Murre \cite{murre2} that every smooth
projective surface $S$ admits a Chow--K\"unneth decomposition
$\Delta_S = \pi_S^0 + \pi_S^1 + \pi_S^2 + \pi_S^3 + \pi_S^4$.

The notion of Chow--K\"unneth decomposition is significant because
when it exists it induces a filtration $\mathrm{F}^l\CH^i(X) :=
\bigoplus_{s\geq l}\CH^i(X)_s$ on the Chow group $\CH^*(X)$ which
should not depend on the choice of the Chow--K\"unneth decomposition
$\Delta_X = \sum_{i=0}^{2d} \pi^i$ and which should be of
Bloch--Beilinson type ;
\emph{cf.} \cite{jannsen, murre}. \\

Let now $S^{[n]}$ denote the Hilbert scheme of length-$n$ subschemes on a
smooth projective surface $S$. By Fogarty \cite{fogarty}, the scheme $S^{[n]}$
is in fact a
smooth projective variety, and it comes equipped with a morphism
$S^{[n]} \rightarrow S^{(n)}$ to the $n^\mathrm{th}$ symmetric product
of $S$, called the \emph{Hilbert--Chow morphism}.  De Cataldo and
Migliorini \cite{dCM} have given an explicit description of the motive
of $S^{[n]}$. Let us introduce some notations related to this
description. Let $\mu = \{A_1, \ldots, A_l\}$ be a partition of the
set $\{1,\ldots, n\}$, where all the $A_i$'s are non-empty. The
integer $l$, also denoted $l(\mu)$, is the \emph{length} of the
partition $\mu$. Let $S^\mu \simeq S^l \subseteq S^n$ be the
set $$\{(s_1,\ldots,s_n) : s_i = s_j \mbox{ if } i,j \in A_k \mbox{
  for some } k\}$$ and let $$\Gamma_\mu := (S^\mu \times_{S^{(n)}}
S^{[n]})_{red} \subset S^\mu \times S^{[n]},$$ where the subscript $red$ means
the  underlying
reduced scheme. It is known that $\Gamma_\mu$
is irreducible of dimension $n+l(\mu)$.  The subgroup
$\mathfrak{S}_\mu$ of $\mathfrak{S}_n$ that acts on $\{1,\ldots,n\}$
by permuting the $A_i$'s with same cardinality acts on the first
factor of the product $ S^\mu \times S^{[n]}$, and the correspondence
$\Gamma_\mu$ is invariant under this action. We can therefore
define $$\hat{\Gamma}_\mu := \Gamma_\mu / \mathfrak{S}_\mu \in
\CH^*(S^{(\mu)} \times S^{[n]}) = \CH^*(S^\mu \times
S^{[n]})^{\mathfrak{S}_\mu} ,$$ where $S^{(\mu)} :=
S^\mu/\mathfrak{S}_\mu$. Since for a variety $X$ endowed with the
action of a finite group $G$ we have $\CH^*(X/G) = \CH^*(X)^G$ (with
rational coefficients), the calculus of correspondences and the theory
of motives in the setting of smooth projective varieties endowed with
the action of a finite group is similar in every way to the usual case
of smooth projective varieties. We will therefore freely consider
actions of correspondences and motives of quotient varieties by the action of a
finite group.

The symmetric groups $\mathfrak{S}_n$ acts naturally on the set of partitions of
$\{1, \ldots, n\}$. By choosing one element in each orbit for the above action,
we may define a subset $\mathfrak{B}(n)$ of the set of partitions of $\{1,
\ldots, n\}$. This set is isomorphic to the set of partitions of the integer
$n$.

\begin{thm}[de Cataldo and Migliorini \cite{dCM}] \label{thm dcm}
Let $S$ be a smooth
  projective surface defined over an arbitrary field. The morphism
\begin{equation}
\label{eq dCM}
\bigoplus_{\mu \in \mathfrak{B}(n)} {}^t\hat{\Gamma}_\mu \ : \
\h(S^{[n]}) 
\stackrel{\simeq}{\longrightarrow} \bigoplus_{\mu \in \mathfrak{B}(n)} 
\h(S^{(\mu)})(l(\mu)-n)
\end{equation}
is an isomorphism of Chow motives. Moreover, its inverse is given by
the correspondence $\sum_{\mu \in \mathfrak{B}(n)} \frac{1}{m_\mu}
\hat{\Gamma}_\mu$ for some non-zero rational numbers $m_\mu$.
\end{thm}

Let now $\Delta_S = \pi_S^0 + \pi_S^1 + \pi_S^2 + \pi_S^3 + \pi_S^4$
be a Chow--K\"unneth decomposition of $S$. For all non-negative integers $m$,
the correspondences
\begin{equation}
\label{eq CKSm}
\pi^i_{S^m} := \sum_{i_1+\ldots + i_m = i} \pi_S^{i_1} \otimes \cdots 
\otimes \pi_S^{i_m} \quad \mbox{in } \CH^{2m}(S^m \times S^m)
\end{equation}
define a Chow--K\"unneth decomposition of $S^m$ that is clearly
$\mathfrak{S}_m$-equivariant. Therefore, these correspondences can be
seen as correspondences of $\CH^{2m}(S^{(m)} \times S^{(m)})$ and they
do define a Chow--K\"unneth decomposition of the $m^\mathrm{th}$
symmetric product $S^{(m)}$.  Let us denote this
decomposition $$\Delta_{S^{(m)}} = \pi^0_{S^{(m)}} + \cdots +
\pi^{4m}_{S^{(m)}} \quad \mbox{in } \CH^{2m}(S^{(m)} \times
S^{(m)}).$$ Since
$S^l$ is endowed with a $\mathfrak{S}_l$-equivariant
Chow--K\"unneth decomposition as above and since $\mathfrak{S}_\mu$ is
a subgroup of $\mathfrak{S}_l$, $S^\mu \simeq S^l$ is endowed with a
$\mathfrak{S}_\mu$-equivariant Chow--K\"unneth
decomposition. Therefore $S^{(\mu)}$ is endowed with a natural
Chow--K\"unneth decomposition $$\Delta_{S^{(\mu)}} = \pi^0_{S^{(\mu)}}
+ \cdots + \pi^{4l}_{S^{(\mu)}} \quad \mbox{in } \CH^{2l}(S^{(\mu)}
\times S^{(\mu)})$$ coming from that of $S$. In particular, the
isomorphism of de Cataldo and Migliorini gives a natural
Chow--K\"unneth decomposition for the Hilbert scheme $S^{[n]}$ coming
from that of $S$. Precisely, this Chow--K\"unneth decomposition is
given by
\begin{equation}
\label{eq CKhilb}
\pi_{S^{[n]}}^i = \sum_{\mu \in \mathfrak{B}(n)} \frac{1}{m_\mu}
\hat{\Gamma}_\mu 
\circ \pi^{i-2n+2l(\mu)}_{S^{(\mu)}} \circ {}^t\hat{\Gamma}_\mu.
\end{equation}
Note that if the Chow--K\"unneth decomposition $\{\pi_S^i\}$ of $S$ is
self-dual, then the Chow--K\"unneth decomposition $\{\pi_{S^{[n]}}^i\}$ of
$S^{[n]}$ is also self-dual. 

We will show that when $S$ is either a K3 surface or an abelian
surface the Chow--K\"unneth decomposition above induces a
decomposition of the motive $\h(S^{[n]})$ that satisfies the
conclusions of Theorem \ref{thm main} for an appropriate choice of
Chow--K\"unneth decomposition for $S$.

\section{Multiplicative Chow--K\"unneth decompositions} \label{sec mult}

Let $X$ be a smooth projective variety of dimension $d$ and let
$\Delta_3 \in \CH_d(X \times X \times X)$ be the small diagonal, that
is, the class of the subvariety $$\{(x,x,x) : x \in X\} \subset X
\times X \times X.$$ If we view $\Delta_3$ as a correspondence from $X
\times X$ to $X$, then $\Delta_3$ induces the \emph{multiplication
  morphism}
\begin{equation} \label{eq mult}
\h(X) \otimes \h(X) \rightarrow \h(X).
\end{equation}
Note that if $\alpha$ and $\beta$ are cycles in $\CH^*(X)$, then
$(\Delta_3)_*(\alpha \times \beta) = \alpha \cdot \beta$ in $\CH^*(X)$.

If $X$ admits a Chow--K\"unneth decomposition
\begin{equation}\label{eq ckh}
\h(X) = \bigoplus_{i=0}^{2d} \h^i(X),
\end{equation}
 then this decomposition is said to be
\emph{multiplicative} if the multiplication morphism $\h^i(X) \otimes
\h^j(X) \rightarrow \h(X)$ factors through $\h^{i+j}(X)$ for all $i$
and $j$. For a variety to be endowed with a multiplicative
Chow--K\"unneth decomposition is very restrictive ; we refer to
\cite{sv}, where this notion was introduced, for some
discussions. Examples of varieties admitting a multiplicative
Chow--K\"unneth decomposition are provided by \cite[Theorem 6]{sv} and
include hyperelliptic curves, K3 surfaces, abelian varieties, and their Hilbert
squares.

If one writes $\Delta_X = \pi_X^0 + \ldots + \pi_X^{2d}$ for the
Chow--K\"unneth decomposition \eqref{eq ckh} of $X$, then by definition this
decomposition is multiplicative if $$\pi_X^k \circ \Delta_3 \circ
(\pi_X^i \otimes \pi_X^j) = 0 \mbox{ in } \CH_d(X^3) \mbox{ for all }
k \neq i+j,$$ or equivalently if
$$({}^t\pi_X^i \otimes {}^t\pi_X^j \otimes \pi_X^k)_*\Delta_3 =  0 
\mbox{ in } \CH_d(X^3) \mbox{ for all } k \neq i+j.$$
 If the Chow--K\"unneth decomposition $\{\pi_X^i\}$ is self-dual, then it is
multiplicative if 
 $$(\pi_X^i \otimes \pi_X^j \otimes \pi_X^k)_*\Delta_3 =  0 
 \mbox{ in } \CH_d(X^3) \mbox{ for all } i+j+k \neq 4d.$$
   Note that the
above three relations always hold modulo homological
equivalence.\medskip

Given a multiplicative Chow--K\"unneth decomposition $\pi^i_S$ for a
surface $S$, one could expect the Chow--K\"unneth decomposition
\eqref{eq CKhilb} of $S^{[n]}$ to be multiplicative. This was answered
affirmatively when $n=2$ for any smooth projective variety $X$ with a self-dual
Chow--K\"unneth decomposition (under some additional assumptions on the Chern
classes of $X$) in \cite{sv}, and a similar result when $n=3$ will appear in
\cite{sv2}. (For $n>3$, $X^{[n]}$ is no longer smooth if $X$ is smooth of
dimension $> 2$). Here we deal with the case when $S$ is a K3 surface or an
abelian surface and will prove Theorem \ref{thm main}. By the
isomorphism \eqref{eq dCM} of de Cataldo and Migliorini, it is enough
to check that
$$(\hat{\Gamma}_{\mu_1} \otimes \hat{\Gamma}_{\mu_2} \otimes
\hat{\Gamma}_{\mu_3})^* (\pi^i_{S^{[n]}} \otimes \pi^j_{S^{[n]}}
\otimes \pi^k_{S^{[n]}})_*\Delta_3 = 0$$ for all $i+j+k \neq 8n$ and
for all partitions $\mu_1, \mu_2$ and $\mu_3$ of $\{1,\ldots,n\}$, or
equivalently for all $i,j,k$ such that $(\pi^i_{S^{[n]}} \otimes
\pi^j_{S^{[n]}} \otimes \pi^k_{S^{[n]}})_*[\Delta_3] = 0$ in
$\HH^*(S^{[n]} \times S^{[n]} \times S^{[n]}, \Q)$ and all partitions
$\mu_1, \mu_2$ and $\mu_3$. By \eqref{eq CKhilb}, it is even enough to
show that
\begin{equation} \label{eq 1} \left(\hat{\Gamma}_{\mu_1} \otimes
    \hat{\Gamma}_{\mu_2} \otimes \hat{\Gamma}_{\mu_3}\right)^* \left(
    \left(\hat{\Gamma}_{\nu_1} \circ \pi^i_{S^{(\nu_1)}} \circ
      {}^t\hat{\Gamma}_{\nu_1} \right) \otimes
    \left(\hat{\Gamma}_{\nu_2} \circ \pi^j_{S^{(\nu_2)}} \circ
      {}^t\hat{\Gamma}_{\nu_2} \right) \otimes
    \left(\hat{\Gamma}_{\nu_3} \circ \pi^k_{S^{(\nu_3)}} \circ
      {}^t\hat{\Gamma}_{\nu_3} \right) \right)_* \Delta_3
\end{equation} 
is zero in $\CH^*(S^{(\mu_1)} \times S^{(\mu_2)} \times S^{(\mu_3)})$
for all partitions $\mu_1, \mu_2, \mu_3$, and all partitions $\nu_1,
\nu_2, \nu_3$ and all $i,j,k$ such that
$$ \left( \left(\hat{\Gamma}_{\nu_1} \circ \pi^i_{S^{(\nu_1)}} \circ 
    {}^t\hat{\Gamma}_{\nu_1} \right) \otimes
  \left(\hat{\Gamma}_{\nu_2} \circ \pi^j_{S^{(\nu_2)}} \circ
    {}^t\hat{\Gamma}_{\nu_2} \right) \otimes
  \left(\hat{\Gamma}_{\nu_3} \circ \pi^k_{S^{(\nu_3)}} \circ
    {}^t\hat{\Gamma}_{\nu_3} \right) \right)_* [\Delta_3] =0 \mbox{ in
} \HH^*((S^{[n]})^3,\Q).$$ Note that the expression \eqref{eq 1} is
equal to
$$ \left[ \left( {}^t\hat{\Gamma}_{\mu_1} \otimes
    {}^t\hat{\Gamma}_{\mu_2} \otimes {}^t\hat{\Gamma}_{\mu_3} \right)
  \circ \left( \hat{\Gamma}_{\nu_1} \otimes \hat{\Gamma}_{\nu_2}
    \otimes \hat{\Gamma}_{\nu_3} \right) \circ \left(
    \pi^i_{S^{(\nu_1)}} \otimes \pi^j_{S^{(\nu_2)}}
    \otimes\pi^k_{S^{(\nu_3)}} \right) \circ \left(
    {}^t\hat{\Gamma}_{\nu_1} \otimes {}^t\hat{\Gamma}_{\nu_2} \otimes
    {}^t\hat{\Gamma}_{\nu_3} \right) \right]_*\Delta_3.$$ 
    But it is clear from Theorem \ref{thm dcm} that
$$\left( {}^t\hat{\Gamma}_{\mu_1} \otimes  {}^t\hat{\Gamma}_{\mu_2} 
  \otimes {}^t\hat{\Gamma}_{\mu_3} \right) \circ \left(
  \hat{\Gamma}_{\nu_1} \otimes \hat{\Gamma}_{\nu_2} \otimes
  \hat{\Gamma}_{\nu_3} \right) = \left\{
  \begin{array}{lr}
    0 & \mbox{if} \ (\nu_1,\nu_2,\nu_3) \neq (\mu_1,\mu_2,\mu_3)\\
    m_{\mu_1} m_{\mu_2} m_{\mu_3} \, \Delta_{S^{(\mu_1)} \times
      S^{(\mu_2)} 
      \times S^{(\mu_3)}} & \mbox{if} \ (\nu_1,\nu_2,\nu_3) 
    = (\mu_1,\mu_2,\mu_3).
  \end{array}
\right. $$ Thus we have proved the following criterion for the
Chow--K\"unneth decomposition \eqref{eq CKhilb} to be multiplicative~:

\begin{prop} \label{prop reduction} The Chow--K\"unneth decomposition
  \eqref{eq CKhilb} is multiplicative (equivalently, the motive of
  $S^{[n]}$ splits as in Theorem \ref{thm main}) if for all partitions
  $\mu_1, \mu_2$ and $\mu_3$ of the set $\{1,\ldots,n\}$ $$\left(
    \pi^i_{S^{\mu_1}} \otimes \pi^j_{S^{\mu_2}}
    \otimes\pi^k_{S^{\mu_3}} \right)_* \left({\Gamma}_{\mu_1} \otimes
    {\Gamma}_{\mu_2} \otimes {\Gamma}_{\mu_3} \right)^*\Delta_3 = 0
  \mbox{ in } \CH^*(S^{\mu_1} \times S^{\mu_2} \times S^{\mu_3})$$ as
  soon as $$\left( \pi^i_{S^{\mu_1}} \otimes \pi^j_{S^{\mu_2}}
    \otimes\pi^k_{S^{\mu_3}} \right)_* \left({\Gamma}_{\mu_1} \otimes
    {\Gamma}_{\mu_2} \otimes {\Gamma}_{\mu_3} \right)^*[\Delta_3] = 0
  \mbox{ in } \HH^*(S^{\mu_1} \times S^{\mu_2} \times S^{\mu_3},\Q). $$
  \vspace{-18pt}
  
  \qed
\end{prop}

\section{Proof of Theorem \ref{thm main} and Theorem \ref{thm bigrading}} 

The proof is inspired by the
proof of Claire Voisin's \cite[Theorem 5.1]{voisin diag}. In fact,  because of
\cite[Proposition 8.12]{sv}, Theorem \ref{thm main} for K3 surfaces implies
\cite[Theorem 5.1]{voisin diag}. The first
step towards the proof of Theorem \ref{thm main} is to understand the
cycle $\left({\Gamma}_{\mu_1} \otimes {\Gamma}_{\mu_2} \otimes
  {\Gamma}_{\mu_3} \right)^*\Delta_3$. The following proposition, due
to Voisin \cite{voisin diag} (see also \cite{voisin2}), builds on the
work of Ellingsrud, G\"ottsche and Lehn \cite{egl}. Here, $S$ is a
smooth projective surface and $\Delta_k$ is the class of the small
diagonal inside $S^k$ in $\CH_2(S^k)$.

\begin{prop} [Voisin, Proposition 5.6 in \cite{voisin diag}] \label{prop voisin}

For any set of partitions $\boldsymbol\mu :=
  \{\mu_1,\ldots,\mu_k \}$ of $\{1,\ldots,n\}$, there exists a
  universal (i.e., independent of $S$) polynomial $P_{\boldsymbol\mu}$
  with the following property~:
$$({\Gamma}_{\mu_1} \otimes \ldots \otimes {\Gamma}_{\mu_k})^*\Delta_k 
= P_{\boldsymbol\mu}(pr_r^*c_2(S), pr_{r'}^*K_S, pr_{s,t}^*\Delta_S)
\quad \mbox{in} \ \CH^*(S^{\boldsymbol\mu}),$$ where the $pr_r$'s are
the projections from $S^{\boldsymbol\mu} := \prod_i S^{\mu_i} \simeq
S^N$ to its factors, and the $pr_{s,t}$'s are the projections from
$S^{\boldsymbol\mu}$ to the products of two of its factors. \qed
\end{prop}

In fact, Proposition \ref{prop voisin} is a particular instance of \cite[Theorem
5.12]{voisin diag}. Another consequence of  \cite[Theorem 5.12]{voisin diag},
which will be used to prove Theorem \ref{thm bigrading}, is

\begin{prop} [Voisin] \label{prop voisin2} 
    For any partition $\mu$ of $\{1,\ldots,n\}$ and any polynomial $P$ in the
Chern classes of $S^{[n]}$, the cycle $(\Gamma_\mu)^*P$ of $S^\mu$ is a
      universal (i.e., independent of $S$) polynomial in the variables
$pr_r^*c_2(S), pr_{r'}^*K_S, pr_{s,t}^*\Delta_S$,  where the $pr_r$'s are
    the projections from $S^\mu\simeq
    S^N$ to its factors, and the $pr_{s,t}$'s are the projections from
    $S^{\mu}$ to the products of two of its factors. \qed
\end{prop}

We first prove Theorems \ref{thm main} \& \ref{thm bigrading} for $S$ a K3
surface and then
for $S$ an abelian surface. Note that clearly a multiplicative Chow--K\"unneth
decomposition $\{\pi^i_{S^{[n]}}, 0\leq i \leq 4n\}$ induces a multiplicative
bigrading on the  Chow ring $\CH^*(S^{[n]})$ :
 $$\CH^*(S^{[n]}) = \bigoplus_{i,s} \CH^i(S^{[n]})_s, \quad \mbox{where } 
\CH^i(S^{[n]})_s = (\pi^{2i-s}_{S^{[n]}})_*\CH^i(S^{[n]}).$$ 
Thus once Theorem \ref{thm main} is established it only remains to show that the
Chern classes of $S^{[n]}$ sit in $\CH^*(S^{[n]})_0$ in order to conclude.

\subsection{The Hilbert scheme of points on a K3 surface} Let $S$ be a
smooth projective surface and let $\o$ be a zero-cycle of degree $1$
on $S$.  Let $m$ be a positive integer and consider the $m$-fold
product $S^m$ of $S$. Let us define the idempotent
correspondences 
\begin{equation} \label{eq CK surface}
\pi_S^0 := pr_1^*\o = \o \times S, \quad \pi_S^{4} =
pr_2^*\o = S \times \o, \quad \mbox{and} \quad \pi_S^2 := \Delta_S -
\pi_S^0-\pi_S^4.
\end{equation}
(Note that $\pi^2_S$ is not quite a Chow--K\"unneth projector, it projects onto
$\HH^1(S,\Q) \oplus \HH^2(S,\Q) \oplus \HH^3(S,\Q)$.)
 In this case, the idempotents $\pi_{S^m}^i :=
\sum_{i_1+\ldots + i_n = i} \pi_S^{i_1} \otimes \cdots \otimes
\pi_S^{i_n}$ given in \eqref{eq CKSm} are clearly sums of monomials of
degree $2m$ in $pr_r^*\o$ and $pr_{s,t}^*\Delta_S$. By Proposition
\ref{prop voisin}, it follows that for any smooth projective surface
$S$ and any zero-cycle $\o$ of degree $1$ on $S$
$$\left( \pi^i_{S^{\mu_1}} \otimes \pi^j_{S^{\mu_2}}
  \otimes\pi^k_{S^{\mu_3}} \right)_* \left({\Gamma}_{\mu_1} \otimes
  {\Gamma}_{\mu_2} \otimes {\Gamma}_{\mu_3} \right)^*\Delta_3$$ is a
polynomial $Q_{\boldsymbol \mu, i, j, k}$ in the variables
$pr_r^*c_2(S), pr_{r'}^*K_S, pr_{r''}^*\o$ and $pr_{s,t}^*\Delta_S$.

We now have the following key result which is due to Claire Voisin
\cite[Corollary 5.9]{voisin diag} and which relies in an essential way
on a theorem due to Qizheng Yin \cite{yin} that describes the
cohomological relations among the cycles $pr_{r,s}^*\pi^2_S$.
\begin{prop}[Voisin \cite{voisin diag}] \label{prop key} For all
  smooth projective surfaces $S$ and any degree-$1$ zero-cycle $\o$ on
  $S$, let $P$ be a polynomial (independent of $S$) in the variables
  $pr_r^*[c_2(S)]$, $pr_{r'}^*[K_S], pr_{r''}^*[\o]$ and
  $pr_{s,t}^*[\Delta_S]$ with value an algebraic cycle of $S^n$. If $P$ vanishes
for all smooth projective
  surfaces with $b_1(S)=0$, then the polynomial $P$ belongs to the
  ideal generated by the relations :
\begin{enumerate}[(a)]
\item $[c_2(S)] = \chi_{top}(S)[\o]$ ;
\item $[K_S]^2 = \deg(K_S^2) [\o]$ ;
\item $[\Delta_S] \cdot pr_1^*[K_S] = pr_1^*[K_S] \cdot pr_2^*[\o] +
  pr_1^*[\o] \cdot pr_2^*[K_S]$ ;
\item $[\Delta_3] = pr_{1,2}^*[\Delta_S]\cdot pr_3^*[\o] +
  pr_{1,3}^*[\Delta_S]\cdot pr_2^*[\o] + pr_{2,3}^*[\Delta_S]\cdot
  pr_1^*[\o] - pr_1^*[\o]\cdot pr_2^*[\o] - pr_1^*[\o]\cdot pr_3^*[\o]
  - pr_2^*[\o]\cdot pr_3^*[\o]$ ;
\item $[\Delta_S]^2 = \chi_{top}(S)\, pr_1^*[\o]\cdot pr_2^*[\o]$ ;
\item $[\Delta_S]\cdot pr_1^*[\o] = pr_1^*[\o]\cdot pr_2^*[\o]$. \qed
\end{enumerate}
\end{prop}

We may then specialize to the case where $S$ is a K3 surface.
Consider then a K3 surface $S$ and let $\o$ be the class of a point
lying on a rational curve of $S$.  Note that by definition of a K3 surface
$K_S=0$. The following theorem of Beauville and Voisin shows that the
relations (a)--(f) listed above actually hold modulo rational
equivalence.
\begin{thm}[Beauville--Voisin \cite{bv}] \label{thm bv} Let $S$ be a
  K3 surface and let $\o$ be a rational point lying on a rational
  curve on $S$. The following relations hold :
 \begin{enumerate}[(i)]
\item $c_2(S) = \chi_{top}(S) \o \ (=24\o)$ in $\CH^2(S)$ ;
\item $\Delta_3 = pr_{1,2}^*\Delta_S\cdot pr_3^*\o +
  pr_{1,3}^*\Delta_S\cdot pr_2^*\o + pr_{2,3}^*\Delta_S\cdot pr_1^*\o
  - pr_1^*\o\cdot pr_2^*\o - pr_1^*\o\cdot pr_3^*\o - pr_2^*\o\cdot
  pr_3^*\o$ in $\CH_2(S \times S \times S)$.
 \end{enumerate}
 \end{thm}
 
 The proof of Theorem \ref{thm main} in the case when $S$ is a K3
 surface is then immediate~: by the discussion above if the cycle
 $\delta:= \left( \pi^i_{S^{\mu_1}} \otimes \pi^j_{S^{\mu_2}}
   \otimes\pi^k_{S^{\mu_3}} \right)_* \left({\Gamma}_{\mu_1} \otimes
   {\Gamma}_{\mu_2} \otimes {\Gamma}_{\mu_3} \right)^*\Delta_3$ is
 zero in $\HH^*(S^{\mu_1} \times S^{\mu_2} \times S^{\mu_3},\Q)$, then by
 Proposition \ref{prop key} it belongs to the ideal generated by the
 relations (a)--(f). By Theorem \ref{thm bv}, the relations (a)--(f)
 actually hold modulo rational equivalence. Therefore, the cycle
 $\delta$ is zero in $\CH^*(S^{\mu_1} \times S^{\mu_2} \times S^{\mu_3})$. We
 may then conclude by invoking Proposition \ref{prop reduction}. 
 
 It remains to prove that the Chern classes $c_i(S^{[n]})$ sit in
$\CH^{i}(S^{[n]})_0$. It suffices to show that $(\pi^j_{S^{[n]}})_*c_i(S^{[n]})
= 0$ in $\CH^i(S^{[n]})$ as soon as $(\pi^j_{S^{[n]}})_*[c_i(S^{[n]})] = 0$ in
$\HH^{2i}(S^{[n]},\Q)$ (equivalently as soon as $j \neq 2i$). 
 By de Cataldo and Migliorini's theorem, it is enough to show for all partitions
$\mu$ of $\{1,\ldots,n\}$ that $(\Gamma_\mu)^*(\pi^j_{S^{[n]}})_*c_i(S^{[n]}) =
0$ in $\CH^*(S^{\mu})$ as soon as
$(\Gamma_\mu)^*(\pi^j_{S^{[n]}})_*[c_i(S^{[n]})] = 0$ in $\HH^{*}(S^{\mu},\Q)$.
  Proceeding as in section \ref{sec mult}, it is even enough to show that, for
all  partitions $\mu$ of $\{1,\ldots,n\}$,
$(\pi^j_{S^{\mu}})_*(\Gamma_\mu)^*c_i(S^{[n]}) = 0$ in $\CH^*(S^{\mu})$ as soon
as $(\pi^j_{S^{\mu}})_*(\Gamma_\mu)^*[c_i(S^{[n]})] = 0$ in
$\HH^{*}(S^{\mu},\Q)$.
  %\footnote{to be checked.}
   By Proposition \ref{prop voisin2}, $(\Gamma_\mu)^*c_i(S^{[n]})$ is a
universal polynomial in the variables $pr_r^*c_2(S), pr_{r'}^*K_S,
pr_{s,t}^*\Delta_S$. It follows that
$(\pi^j_{S^{\mu}})_*(\Gamma_\mu)^*c_i(S^{[n]})$ is also a universal polynomial
in the variables $pr_r^*c_2(S), pr_{r'}^*K_S, pr_{s,t}^*\Delta_S$.
   We can then conclude thanks to Proposition \ref{prop key} and Theorem
\ref{thm bv}.  \qed

\subsection{The Hilbert scheme of points on an abelian surface}
Let $A$ be an abelian surface. In that case, the Chow--K\"unneth
projectors $\{\pi_A^i\}$ given by the theorem of Deninger and Murre are
\emph{symmetrically distinguished} in the Chow ring $\CH^*(A \times A)$ in
the sense of O'Sullivan \cite{o'sullivan}. (We refer to \cite[Section 7]{sv} for
a summary of O'Sullivan's theory of symmetrically distinguished cycles on
abelian varieties.) Let us mention that the identity element $O_A$ of $A$ plays
the role of the Beauville--Voisin cycle $\o$ in the case of K3 surfaces,
\emph{e.g.} $\pi^0_A = O_A \times A$.
By O'Sullivan's theorem, the Chow--K\"unneth projectors $\pi^i_{A^m}$ given in
\eqref{eq CKSm} are symmetrically distinguished  for all positive
integers $m$. By Proposition
\ref{prop voisin}, the cycle $\left({\Gamma}_{\mu_1} \otimes
  {\Gamma}_{\mu_2} \otimes {\Gamma}_{\mu_3} \right)^*\Delta_3$ is a
polynomial in the variables $pr_r^*c_2(A)$, $pr_{r'}^*K_A $ and
$pr_{s,t}^*\Delta_A$. Since $c_2(A) = 0 $ and $ K_A=0$, this cycle is
in fact symmetrically distinguished. It immediately follows
that $$\left( \pi^i_{A^{\mu_1}} \otimes \pi^j_{A^{\mu_2}}
  \otimes\pi^k_{A^{\mu_3}} \right)_* \left({\Gamma}_{\mu_1} \otimes
  {\Gamma}_{\mu_2} \otimes {\Gamma}_{\mu_3} \right)^*\Delta_3$$ is
symmetrically distinguished. Thus by O'Sullivan's theorem
\cite{o'sullivan}, this cycle is rationally trivial if and only if it
is numerically trivial. By Proposition \ref{prop reduction}, we
conclude that $A^{[n]}$ has a multiplicative Chow--K\"unneth
decomposition. The proof of Theorem \ref{thm main} is now
complete. 

 It remains to prove that the Chern classes $c_i(A^{[n]})$ sit in
$\CH^{i}(A^{[n]})_0$. As in the case of K3 surfaces, it suffices to show that,
for all  partitions $\mu$ of $\{1,\ldots,n\}$,
$(\pi^j_{A^{\mu}})_*(\Gamma_\mu)^*c_i(A^{[n]}) = 0$ in $\CH^*(A^{\mu})$ as soon
as $(\pi^j_{A^{\mu}})_*(\Gamma_\mu)^*[c_i(A^{[n]})] = 0$ in
$\HH^{*}(A^{\mu},\Q)$.
By Proposition \ref{prop voisin2}, $(\Gamma_\mu)^*c_i(A^{[n]})$ is a 
polynomial in the variables $pr_r^*c_2(A) = 0, pr_{r'}^*K_A = 0,
pr_{s,t}^*\Delta_A$. It follows that the cycle
$(\pi^j_{A^{\mu}})_*(\Gamma_\mu)^*c_i(A^{[n]})$ is symmetrically distinguished.
    We can then conclude thanks to O'Sullivan's theorem.  \qed

\section{Decomposition theorems for the relative Hilbert scheme of
  abelian surface schemes and of families of K3 surfaces}

In this section, we generalize Voisin's decomposition theorem
\cite[Theorem 0.7]{voisin k3} for families of K3 surfaces to families
of Hilbert schemes of points on K3 surfaces or abelian surfaces.

Let $\pi : \mathcal{X} \rightarrow B$ be a smooth projective
morphism. Deligne's decomposition theorem states the following~:
\begin{thm}[Deligne \cite{deligne}]
  In the derived category of sheaves of $\Q$-vector spaces on $B$,
  there is a decomposition (which is non-canonical in general)
\begin{equation} \label{eq deligne}
R\pi_*\Q \cong \bigoplus_i R^i\pi_*\Q[-i].
\end{equation}
\end{thm}
Both sides of \eqref{eq deligne} carry a cup-product~: on the
right-hand side the cup-product is the direct sum of the usual
cup-products $R^i\pi_*\Q \otimes R^j\pi_*\Q \rightarrow R^{i+j}\pi_*\Q
$ defined on local systems, while on the left-hand side the derived
cup-product $R\pi_*\Q \otimes R\pi_*\Q \rightarrow R\pi_*\Q $ is such
that it induces the usual cup-product in cohomology. As explained in
\cite{voisin k3}, the isomorphism \eqref{eq deligne} does not respect
the cup-product in general. Given a family of smooth projective
varieties $\pi : \mathcal{X} \rightarrow B$, Voisin \cite[Question
0.2]{voisin k3} asked if there exists a decomposition as in \eqref{eq
  deligne} which is multiplicative, \emph{i.e.}, which is compatible
with cup-product. By Deninger--Murre \cite{dm}, there does exist such
a decomposition for an abelian scheme $\pi : \mathcal{A} \rightarrow
B$. The main result of \cite{voisin k3} is~:

\begin{thm}[Voisin \cite{voisin k3}] \label{thm dec voisin} For any
  smooth projective family $\pi : \mathcal{X} \rightarrow B$ of K3
  surfaces, there exist a decomposition isomorphism as in \eqref{eq
    deligne} and a nonempty Zariski open subset $U$ of $B$, such that
  this decomposition becomes multiplicative for the restricted family
  $\pi|_U : \mathcal{X}|_U \rightarrow U$.
\end{thm}

Our main result in this section is the following extension of Theorem
\ref{thm dec voisin}~:
\begin{thm} \label{thm dec} Let $\pi : \mathcal{X} \rightarrow B$ be
  either an abelian surface over $B$ or a smooth projective family of
  K3 surfaces. Consider $\pi^{[n]} : \mathcal{X}^{[n]} \rightarrow B$
  the relative Hilbert scheme of length-$n$ subschemes on $\mathcal{X}
  \rightarrow B$. Then there exist a decomposition isomorphism for
  $\pi^{[n]} : \mathcal{X}^{[n]} \rightarrow B$ as in \eqref{eq
    deligne} and a nonempty Zariski open subset $U$ of $B$, such that
  this decomposition becomes multiplicative for the restricted family
  $\pi^{[n]}|_U : \mathcal{X}^{[n]}|_U \rightarrow U$.
\end{thm}
\begin{proof}
The proof follows the original approach of Voisin \cite{voisin k3} (after
reinterpreting, as in \cite[Proposition 8.14]{sv}, the vanishing of the modified
diagonal cycle of Beauville--Voisin \cite{bv} as the multiplicativity of the
Beauville--Voisin Chow--K\"unneth decomposition).

  First, we note that there exist a nonempty Zariski open subset $U$
  of $B$ and relative Chow--K\"unneth projectors $\Pi^i :=
  \Pi^i_{\mathcal{X}^{[n]}|_U/U} \in \CH^{2n}(\mathcal{X}^{[n]}|_U
  \times_U \mathcal{X}^{[n]}|_U)$, which means that
  $\Delta_{\mathcal{X}|_U/U} = \sum_i \Pi^i$, $\Pi^i \circ \Pi^i =
  \Pi^i$, $\Pi^i \circ \Pi^j = 0$ for $i\neq j$, and $\Pi^i$ acts as
  the identity on $R^i(\pi^{[n]}|_U)_*\Q$ and as zero on
  $R^j(\pi^{[n]}|_U)_*\Q$ for $j\neq i$. Indeed, Let $X$ be the
  generic fiber of $\pi : \mathcal{X} \rightarrow B$. If $X$ is a K3
  surface, then we consider the degree $1$ zero-cycle $\o :=
  \frac{1}{24} c_2(X) \in \CH_0(X)$. We then have a Chow--K\"unneth
  decomposition for $X$ given by $\pi^0_X := pr_1^*\o, \pi^4_X :=
  pr_2^*\o$ and $\pi_X^2 := \Delta_X - \pi_X^0 - \pi_X^4$. If $X$ is
  an abelian surface, we may consider the Chow--K\"unneth
  decomposition of Deninger--Murre \cite{dm}. In both cases, these
  Chow--K\"unneth decompositions induce as in \eqref{eq CKhilb} a
  Chow--K\"unneth decomposition $\Delta_{X^{[n]}} = \sum_i
  \pi^i_{X^{[n]}}$ of the Hilbert scheme of points $X^{[n]}$. By
  spreading out, we obtain the existence of a sufficiently small but
  nonempty open subset $U$ of $B$ such that this Chow--K\"unneth
  decomposition spreads out to a relative Chow--K\"unneth
  decomposition $\Delta_{\mathcal{X}|_U/U} = \sum_i \Pi^i$.
  
  By \cite[Lemma 2.1]{voisin k3}, the relative idempotents $\Pi^i$
  induce a decomposition in the derived category $R\pi_*\Q \cong
  \bigoplus_{i=0}^{4n} \HH^i(R\pi_*\Q)[-i] = \bigoplus_{i=0}^{4n}
  R^i\pi_*\Q[-i]$ with the property that $\Pi^i$ acts as the identity
  on the summand $\HH^i(R\pi_*\Q)[-i]$ and acts as zero on the
  summands $\HH^j(R\pi_*\Q)[-j]$ for $j \neq i$.  Thus, in order to
  show the existence of a decomposition as in \eqref{eq deligne} that
  is multiplicative, it is enough to show, up to further shrinking
  $U$ if necessary, that the relative Chow--K\"unneth decomposition $\{\Pi^i\}$
above
  satisfies
  \begin{equation} \label{eq multderived} \Pi^k \circ \Delta_3 \circ
    (\Pi^i \otimes \Pi^j) = 0 \quad \mbox{in} \
    \CH^{4n}((\mathcal{X}^{[n]} \times_B \mathcal{X}^{[n]} \times_B
    \mathcal{X}^{[n]})|_U) \quad \mbox{for all} \ k\neq i+j.
\end{equation}
Here, $\Delta_3$ is the class of the relative small diagonal inside
$\CH^{4n}(\mathcal{X}^{[n]} \times_B \mathcal{X}^{[n]} \times_B
\mathcal{X}^{[n]})$. But then, by Theorem \ref{thm main}, the relation
\eqref{eq multderived} holds generically. Therefore, by spreading out,
\eqref{eq multderived} holds over a nonempty open subset of $B$. This
concludes the proof of the theorem.
\end{proof}

%%%%%%%%%%%%%%%%%%%%%%%%%%%%%%%%%%%%%%%%%%%%%%%%%%%%%%%%%%%%%%%%%%%%%%%%%%%%%%
%%%%%%%%%%%%%%%%%%%%%%%%%%%%%%%%%%%%%%%%%%%%%%%%%%%%%%%%%%%%%%%%%%%%%%%%%%%%%%
%%%               BIBLIOGRAPHY                                           %%%%%
%%%%%%%%%%%%%%%%%%%%%%%%%%%%%%%%%%%%%%%%%%%%%%%%%%%%%%%%%%%%%%%%%%%%%%%%%%%%%%
%%%%%%%%%%%%%%%%%%%%%%%%%%%%%%%%%%%%%%%%%%%%%%%%%%%%%%%%%%%%%%%%%%%%%%%%%%%%%%

%\addcontentsline{toc}{part}{References}

\end{document}